\documentclass[11pt]{article}
\usepackage{geometry}                		
\usepackage[parfill]{parskip}    		
\usepackage{graphicx}				
\usepackage{subfig}								

\usepackage{amssymb}
\usepackage{amsmath}

\usepackage{palatino}

\usepackage{graphicx}
\usepackage[english]{babel}
\usepackage{amscd}
\usepackage{amsthm}

\newtheorem{theorem}{Theorem}[section]
\newtheorem{lemma}[theorem]{Lemma}

\newtheorem{corollary}[theorem]{Corollary}

\title{Lower bounds on binomial and Poisson tails: an approach via tail conditional expectations}

\author{Christos Pelekis\thanks{The Czech Academy of Sciences, Institute of Computer Science, Pod Vod\'{a}renskou v\v{e}\v{z}\'{\i} 2, 182 07 Prague, Czech Republic. E-mail: pelekis.chr@gmail.com} }

\begin{document}

\maketitle

\begin{abstract}  We derive upper bounds on the tail conditional expectation of binomial and Poisson random variables.  
Those upper bounds are subsequently employed to the problem of obtaining  non-asymptotic lower bounds on the probability that the aforementioned  random variables are  
significantly larger than their expectation. 
\end{abstract}



\section{Prologue and main result}

Given a positive integer $n$ and a real number $p\in (0,1)$, we denote by $\text{Bin}(n,p)$ a binomial 
random variable of parameters $n$ and $p$. Given a positive real number $\mu$, we denote by $\text{Poi}(\mu)$ a Poisson random variable of mean $\mu$. 
Throughout the text, given two random variables $X,Y$, the notation 
$X\sim Y$ will indicate that they have the same distribution. 
 In this article we shall be interested in \emph{lower bounds} on 
the "tail" probabilities $\mathbb{P}\left[X\geq k\right]$, where $X\sim \text{Bin}(n,p)$ and $k>np$, and $\mathbb{P}\left[Y\ge k\right]$, where $Y\sim \text{Poi}(\mu)$ and $k>\mu$.  

The problem of estimating the tail of a binomial random variable is arguably one of the most basic problems in 
probability and statistics. 
It arises in a plethora of problems and, despite the fact that the binomial tails are asymptotically 
well understood  (see, for example, \cite{Bahadur, Carter, Fellerzero, Littlewood, Massart, McKay, Slud, Zubkov}), 
several questions coming from a variety of disciplines  (see \cite{Alon, Ash, Greenberg, Jerabek, Rigollet}) 
do not decrease the need for upper and lower 
bounds on $\mathbb{P}\left[\text{Bin}(n,p)\geq k\right]$  that are valid for fixed $n,p,k$.  
Sharp upper bounds on binomial tails are provided by the, so-called, Chernoff-Hoeffding bound (see \cite{Hoeffdingone})  
by employing an "exponential moment" approach: 
\begin{equation}\label{Hoebound}
\mathbb{P}\left[\text{Bin}(n,p)\geq k\right] \leq \text{exp}\left(-n D(k/n||p)\right), \; \text{for}\; k>np, 
\end{equation}
where $D(k/n||p)$ is the Kullback-Leibler distance between two Bernoulli $0/1$ random variables of parameter $k/n$ and $p$, respectively. 
Upper bounds that are sharper than the previous one can be obtained by employing  
a "factorial moment" approach (see \cite{Schmidt}): 

\begin{equation}\label{HoeboundSiegel}
\mathbb{P}\left[\text{Bin}(n,p)\geq k\right] \leq p^{\ell} \binom{n}{\ell}/\binom{k}{\ell}, \; \text{where}\;
\ell = \big\lceil  \frac{k-np}{1-p} \big\rceil , \; k>np.
\end{equation}
Here and later, $\lceil x\rceil$ denotes the minimum integer that is larger than or equal to $x$. 
Elementary, though quite tedious, calculations (see \cite{Schmidt} for details) show that the bound
given by (\ref{HoeboundSiegel}) is smaller than the bound given by (\ref{Hoebound}) but the two 
bounds are quite close to each other.

A well-known theorem of Cram\'er implies that the bound of Eq.(\ref{Hoebound}) is asymptotically tight and therefore one can only hope for sharp  lower bounds on the 
binomial tails that are a portion of the 
Chernoff-Hoeffding bound. 
In that regard, one has the following result (see \cite[Lemma $4.7.2$]{Ash}):  
\begin{equation}\label{Ashbound}
\mathbb{P}\left[\text{Bin}(n,p)\geq k\right]\geq \frac{1}{\sqrt{8k(1-k/n)}}\cdot \text{exp}\left(-n D(k/n||p)\right), \; \text{for}\; k>np. 
\end{equation}

Briefly, this bound      
follows from the observation 
$\mathbb{P}\left[\text{Bin}(n,p)\geq k\right]\geq \binom{n}{k}p^k(1-p)^{n-k}$ combined with elaborate use of entropy estimates on binomial coefficients and appears 
to be among the best known lower bounds. 
Several lower bounds on the tail of binomial random variables  
can be found in \cite{Diaconis, Doerr, Greenberg, Jerabek, PelekisRamon, Telgarsky}, among many others. 
Let us remark that most lower bounds in the literature appear to be 
obtained using ideas from asymptotic analysis and/or rely on estimates of the 
normal approximation to the binomial distribution.  In this article we provide yet another lower 
bound on binomial tails which results in a proportion of the bound given by (\ref{HoeboundSiegel}). 
We embark on a different 
approach that employs properties of tail conditional expectations.

In the case of Poisson random variables, the Chernoff-Hoeffding bound provides an upper estimate on the probability that a $\text{Poi}(\mu)$ random variable is significantly larger than its mean and reads as follows (see \cite[page 97]{Mitzenmacher}):
\begin{equation}\label{Poi_upper_bound}
\mathbb{P}\left[\text{Poi}(\mu) \ge k\right] \le \frac{e^{\mu}(e\mu)^k}{k^k}, \; \text{for} \; k >\mu.  
\end{equation}
Several upper bounds on Poisson tails can be found in \cite{Glynn, Teicher}, among others. 
However, the problem of obtaining lower bounds on Poisson tails seems to be less investigated and we were not able to find a systematic treatment of this topic. 

Let us proceed by introducing our main results. 
Notice that the trivial estimates 
\[ p^k \leq   \mathbb{P}\left[\text{Bin}(n,p)\geq k\right] \leq  \binom{n}{k}p^k\] 
imply that there exists a constant, $C_{n,p,k}$, that depends on $n,p,k$ such that 
\[ \mathbb{P}\left[\text{Bin}(n,p)\geq k\right] = C_{n,p,k} \cdot p^k . \]
It turns out that $C_{n,p,k}$ can be described in terms of tails conditional expectations of certain binomial 
random variables. More precisely, we will see in the next section that 
\[ C_{n,p,k} = \prod_{j=0}^{k-1} \frac{n-j}{\mathbb{E}\left[X_{n-j}| X_{n-j}\geq k-j\right]}, \; \text{where} \; 
X_{n-j} \sim \text{Bin}(n-j,p) \] 
and so the problem of obtaining lower bounds on the tail probability is reduced to the problem 
of obtaining upper bounds on tail conditional expectations of binomial random variables.  
The main novelty of our paper is that it appears to be the first report that investigates upper  
bounds on tail conditional expectations as a tool to obtain lower bounds on tail probabilities. 
Let us also remark that the problem of obtaining estimates on the tail conditional expectation (a.k.a. \emph{conditional value at risk}) of spesific distributions is of particular interest in actuarial sciences (see, for example,  \cite{Landsman, WangZhao}). 

Our bounds on the tail conditional expectation and the tail probability of a binomial random variable read as follows. 
Here and later, given a real number $x$, we will denote by $\lfloor x \rfloor$ the maximum integer that 
is less than or equal to $x$.\\

\begin{theorem}\label{mainthm} Let $X \sim \text{Bin}(n,p)$ and fix positive integer $k$ such that $np < k \leq n$. 
Then  
\[ \mathbb{E}\left[X | X\geq k\right] \leq  k + \frac{(n-k)p}{k-np+p}  . \]
Furthermore, if $k$ satisfies $np<k\le n-1$ we have 
\[ \mathbb{P}\left[X_n \geq k\right]   \geq  \frac{p^{2(\ell+1)}}{2}\cdot  \binom{n}{\ell +1}/\binom{k}{\ell +1} ,\]
where $\ell := \lfloor \frac{k-np}{1-p} \rfloor <k$.
\end{theorem}

Notice that when $k$ is such that $np<k< np + 1-p$, then $\ell =0$ and the lower bound of the previous result reduces to 
\begin{equation}\label{eqq:1}
 \mathbb{P}\left[X\ge np\right] \ge \frac{1}{2}\cdot p^2\cdot\frac{n}{k}\ge \frac{p}{2}\cdot \left(1- \frac{1-p}{k}\right)\ge \frac{p}{2}\cdot \left(1- \frac{1}{k}\right) 
\end{equation} 
and the later bound is larger than $\frac{1}{4}\cdot \frac{9}{10}$, when $p>1/2$ and $k\ge 10$. We refer the reader to \cite{Doerr, Greenberg, PelekisRamon}  for sharper lower  bounds on the the tail probability in (\ref{eqq:1}). 
In other words, the second statement of Theorem \ref{mainthm} provides a lower bound on  
binomial tails that is equal to  a particular proportion  of the upper bound given by (\ref{HoeboundSiegel}). 
We prove  Theorem \ref{mainthm} in  Section \ref{sec:2}. Let us remark that comparisons between  
the lower bound provided by Theorem \ref{mainthm} and the  lower 
bound given by (\ref{Ashbound}) require quite tedious calculations. However, 
it is rather straightforward 
to put the computer to work and see that  the bound of   
Theorem \ref{mainthm} sometimes performs better than the 
bound given by (\ref{Ashbound}). We provide some pictorial comparisons between the 
two bounds in Section \ref{compare}.   

Finally, in Section \ref{sec:3}, we show that similar ideas apply to Poisson random variables. 
In particular, we obtain the following. \\

\begin{theorem}\label{poibound}
Let $Y\sim \text{Poi}(\mu)$, where $\mu>0$, and fix a positive integer $k$ such that $k\ge\mu$. 
Then
\[ \mathbb{E}\left[Y|Y\ge k\right] \le  k +\frac{\mu}{k+1-\mu}. \] 
Furthermore, if we set $\ell = \lfloor k-\mu \rfloor$ we have 
\[ \mathbb{P}\left[Y \ge k\right] \ge \frac{1}{2}\cdot \left(\frac{\mu}{k+\mu} \right)^{\ell+1} . \]
\end{theorem}

Notice that when $k$ satisfies $\mu\le k < \mu +1$ then we have $k+\mu <2\mu+1$ as well as $\ell =0$ and the  lower bound of the previous  result reduces to 
\begin{equation}\label{eqq:2} 
\mathbb{P}\left[Y\ge \mu\right] \ge \frac{1}{2}\cdot \frac{\mu}{2\mu+1} , \; \text{for}\; Y\sim \text{Poi}(\mu),  
\end{equation}
and the right hand side is larger than $1/5$, when $\mu\ge 2$.

\section{Proof of  Theorem \ref{mainthm}}\label{sec:2}

In this section we prove  Theorem \ref{mainthm}. 
Let us begin by recalling the following well-known result regarding the median of a binomial random variable.  \\

\begin{theorem}[see \cite{Kaas}]
\label{kaas} Let $X \sim \text{Bin}(n,p)$. Then 
$\mathbb{P}\left[ X\geq \lfloor np\rfloor \right] \geq 1/2$.
\end{theorem}

A basic ingredient in the proof of Theorem \ref{mainthm} is the following relation 
between the tails of a binomial random variable and its tail conditional expectation. \\

\begin{lemma}\label{basiclem} Fix a positive integer $n$ and a real number $p\in (0,1)$. 
For $j=1,\ldots,n$ let $X_j$ be a $\text{Bin}(j,p)$ random variable. 
If $k$ is a positive integer such that $1\leq k \leq n-1$.  Then 
\[ \mathbb{P}\left[X_n\geq k \right] = \frac{np}{\mathbb{E}\left[X_n | X_n\geq k\right]} \cdot 
\mathbb{P}\left[X_{n-1}\geq k-1\right] . \] 
\end{lemma}
\begin{proof} 
Since $X_n = \sum_i B_i$ is a sum of $n$ Bernoulli $0/1$ random variables $B_i$ of mean $p$, the linearity of expectation yields  
\begin{eqnarray*} \mathbb{E}\left[X_n | X_n\geq k\right] = \sum_{i=1}^{n} \mathbb{P}\left[B_i =1 | X_n\geq k\right] =  \frac{n p\cdot \mathbb{P}\left[X_{n-1}\geq k-1\right]}{\mathbb{P}\left[X_{n}\geq k\right]}  
\end{eqnarray*}
and the result follows. 
\end{proof}

Notice that Lemma \ref{basiclem} can be iterated. This provides a way to express the tail 
probability of a binomial random variable in terms of  tail conditional expectations. \\

\begin{corollary}\label{cormain} Fix a positive integer $n$ and a real number $p\in (0,1)$.    
For every $j\leq n$, let $X_j$ be a $\text{Bin}(j,p)$ random variable.  If $\ell$ is any  integer from the set 
$\{0,1,\ldots,k-1\}$ then 
\[ \mathbb{P}\left[X_n\geq k\right] = \mathbb{P}\left[X_{n-(\ell +1)} \geq k-(\ell+1)\right] \cdot  \prod_{j=0}^{\ell} \frac{p(n-j)}{\mathbb{E}\left[X_{n-j} \big| X_{n-j}\geq k-j\right]} 
 .\]
\end{corollary}
\begin{proof} Immediate upon iterating Lemma \ref{basiclem}  $\ell +1$ times. 
\end{proof}

Hence lower bounds on the tail probability can be provided via  
upper bounds on the tail conditional expectations 
$\mathbb{E}\left[X_{n-j} \big| X_{n-j}\geq k-j\right]$, for $j=0,\ldots,k-1$. 
The next result provides a recursion for tail conditional expectations of integer valued random variables 
and is interesting on its own. \\

\begin{lemma}\label{secondlem}
Let $X$ be an integer valued random variable and fix a positive integer $k$ such that 
$\mathbb{P}\left[X\geq k\right]>0$ and $\mathbb{P}\left[X\geq k+1\right]>0$ . 
Then 
\[ \mathbb{E}\left[X| X\geq k\right] = k \cdot \frac{\mathbb{P}\left[X=k\right]}{\mathbb{P}\left[X\geq k\right]} +  \mathbb{E}\left[X| X\geq k+1\right] \cdot \left(1- \frac{\mathbb{P}\left[X=k\right]}{\mathbb{P}\left[X\geq k\right]}\right)   .\]
\end{lemma}
\begin{proof}
The result follows upon writing  
\begin{eqnarray*}  \mathbb{E}\left[X| X\geq k\right] &=& k\cdot \frac{\mathbb{P}\left[X=k\right]}{\mathbb{P}\left[X\geq k\right]} + \sum_{j\ge k+1} j\cdot \frac{\mathbb{P}\left[X=j\right]}{\mathbb{P}\left[X\geq k+1\right]}  
\cdot \frac{\mathbb{P}\left[X\geq k+1\right]}{\mathbb{P}\left[X\geq k\right]} \\
&=& k \cdot \frac{\mathbb{P}\left[X=k\right]}{\mathbb{P}\left[X\geq k\right]} +  \mathbb{E}\left[X| X\geq k+1\right] \cdot \left(1- \frac{\mathbb{P}\left[X=k\right]}{\mathbb{P}\left[X\geq k\right]}\right) .
\end{eqnarray*} 
\end{proof}

We will also need the following result which provides estimates on ratios of binomial tails. \\

\begin{lemma}\label{feller} Let $X_n \sim \text{Bin}(n,p)$, for $n=1,2,\ldots$. Then for every positive integer $k$ such that  $np<k\leq n$ we have
\[ \mathbb{P}\left[X_n \geq k\right] \leq \frac{k(1-p)}{k-np}\cdot \mathbb{P}\left[X_n = k\right] .  \]
In particular, 
\[ \frac{\mathbb{P}\left[X_n \geq k+1\right]}{\mathbb{P}\left[X_n \geq k\right]} \leq \frac{p(n-k)}{k(1-p)}.  \]
\end{lemma}
\begin{proof} The first statement is a well known result. See, for example, Feller \cite[p. $151$]{Feller} for a proof that employs comparison with a geometric series, or 
Diaconis et al. \cite[Theorem $1$]{Diaconis} for a proof that employs the, so-called, Todhunter's formula. 
The second statement is equivalent to the first and follows upon observing that 
\[  \mathbb{P}\left[X_n \geq k+1\right] = \mathbb{P}\left[X_n \geq k\right] - \mathbb{P}\left[X_n = k\right] . \]
For the sake of completeness, let us provide yet another proof of this result that employs Lemma 
\ref{basiclem}. 
Notice that, if we apply Lemma \ref{basiclem} with $n$ replaced by $n+1$ and $k$ replaced by 
$k+1$, we get  
\begin{eqnarray*} 0\leq   \mathbb{E}\left[\max\{0,X_{n+1}-(k+1) \} \right] &=& \mathbb{E}\left[X_{n+1} -(k+1) | X_{n+1}\geq k+1\right] 
\cdot \mathbb{P}\left[X_{n+1}\geq k+1\right] \\
&=& (n+1)p \cdot \mathbb{P}\left[ X_{n}\geq k \right] - (k+1)\cdot \mathbb{P}\left[X_{n+1} \geq k+1 \right] \\
&=& p(n-k)\cdot \mathbb{P}\left[ X_{n}\geq k \right] - (k+1)(1-p) \cdot \mathbb{P}\left[X_{n} \geq k+1 \right].  
\end{eqnarray*}  
Hence we have shown that 
\[ \frac{\mathbb{P}\left[X_n \geq k+1\right]}{\mathbb{P}\left[X_n \geq k\right]} \leq \frac{p(n-k)}{(k+1)(1-p)}  \] 
and the result follows. 
\end{proof}

We can now proceed with an upper bound on the tail conditional expectation of a binomial random variable. 

\begin{proof}[Proof of Theorem \ref{mainthm}(First statement)]  
Clearly, $\mathbb{E}\left[X | X\geq n\right] = n$. Moreover, elementary calculations show that 
\[ \mathbb{E}\left[X | X\geq n-1\right] = n-1 + \frac{p}{(n-1)(1-p)+p}  \leq n-1 + \frac{p}{n-1-np+p}  \]
and so the result holds true for $k\in \{n-1,n\}$. 
For the remaining values, i.e.  $k\leq n-2$, we prove the result by reverse induction on $k$. 
More precisely, we prove that the inequality holds true for $k>np$ under the assumption that  it
holds true for $k+1$. The assumption 
that the inequality holds true for $k+1$ implies that 
\[\mathbb{E}\left[X | X\geq k+1\right] \leq  k+1 + \frac{(n-k-1)p}{k+1-np+p} . \] 
Now notice that the previous inequality combined with  Lemma \ref{secondlem} yield 
\begin{eqnarray*}  \mathbb{E}\left[X| X\geq k\right] &=& k \cdot \frac{\mathbb{P}\left[X=k\right]}{\mathbb{P}\left[X\geq k\right]} +  \mathbb{E}\left[X| X\geq k+1\right] \cdot \frac{\mathbb{P}\left[X\geq k+1\right]}{\mathbb{P}\left[X\geq k\right]}  \\
&\leq& k + \left(1 + \frac{(n-k-1)p}{k+1-np+p}\right)\cdot \frac{\mathbb{P}\left[X\geq k+1\right]}{\mathbb{P}\left[X\geq k\right]} \\
&\leq& k + \left(1 + \frac{(n-k-1)p}{k+1-np+p}\right) \cdot\frac{p(n-k)}{k(1-p)} ,
\end{eqnarray*}
where the last estimate follows from Lemma \ref{feller}. 
Now notice that  the result will follow once we show  
\[ \left(1 + \frac{(n-k-1)p}{k+1-np+p}\right) \cdot\frac{1}{k(1-p)} \leq  \frac{1}{k-np+p}  ,  \]
or, equivalently, once we show 
\[ \frac{k+1-kp}{k+1-np+p } \cdot \frac{1}{k(1-p)}  \leq \frac{1}{k-np+p}.  \] 
Since we assume $k< n-1$, it is now easy to verify that the last inequality holds true and the proof is complete. 
\end{proof}

We are now ready to prove the second statement of Theorem \ref{mainthm}. The proof combines the previous upper bound on 
tail conditional expectations with Corollary \ref{cormain}.

\begin{proof}[Proof of Theorem \ref{mainthm} (Second statement)] 
Notice that $\ell = \lfloor \frac{k-np}{1-p} \rfloor$ is the maximum 
integer $j$ such that $k-j \geq (n-j)p$. 
From Corollary \ref{cormain} we know that 
\[ \mathbb{P}\left[X_n\geq k\right] = \mathbb{P}\left[X_{n-(\ell +1)} \geq k-(\ell+1)\right] \cdot  \prod_{j=0}^{\ell} \frac{p(n-j)}{\mathbb{E}\left[X_{n-j} \big| X_{n-j}\geq k-j\right]}  , 
\]
where $X_{n-j}\sim \text{Bin}(n-j,p)$, for $j=0,\ldots, \ell$. 
Since $k-(\ell +1)$ is a positive integer such that $k-(\ell +1) < (n-(\ell +1))p$, 
Theorem \ref{kaas} yields   
\[ \mathbb{P}\left[X_{n-(\ell +1)} \geq k-(\ell+1)\right]  \geq \frac{1}{2} . \]
Now notice that the first statement of  Theorem \ref{mainthm} combined with the fact that $j\le \ell <k$ imply 
\[ \mathbb{E}\left[X_{n-j} \big| X_{n-j}\geq k-j\right] \le k-j + \frac{n-k}{k-j-(n-j)p+p}  \le k-j + \frac{n-j}{k-j-(n-j)p+p}  \]
and therefore we have  
\[  \frac{p(n-j)}{\mathbb{E}\left[X_{n-j} \big| X_{n-j}\geq k-j\right]}  \geq   \frac{(n-j)\cdot p\cdot (k-j-(n-j)p)+ (n-j)p^2  }{(k-j)\cdot (k-j-(n-j)p)+(n-j)p} ,\; \text{for}\; j=0,\ldots, \ell .\] 
Since $k-j \geq (n-j)p$, for $j=0,\ldots,\ell$,  
straightforward calculations show that   
\[ \frac{(n-j)\cdot p\cdot (k-j-(n-j)p)+ (n-j)p^2  }{(k-j)\cdot (k-j-(n-j)p)+(n-j)p} \geq \frac{n-j}{k-j}\cdot p^2 . \] 
Putting all the above together, we conclude that 
\[ \mathbb{P}\left[X_n\geq k\right] \geq \frac{1}{2}\cdot p^{2(\ell+1)}\prod_{j=0}^{\ell} \frac{n-j}{k-j}  \]
and the result follows. 
\end{proof}

\section{Proof of Theorem \ref{poibound}}\label{sec:3} 

In this section we prove Theorem \ref{poibound}.  We begin with a result that expresses the tail conditional expectation of a Poisson random variable in terms of a ratio between consecutive tails. \\

\begin{lemma}\label{poilem:1}
Let $Y\sim\text{Poi}(\mu)$, where $\mu>0$, and fix a positive integer $k$. 
Then 
\[\mathbb{P}\left[Y\ge k\right] = \mu\cdot \frac{\mathbb{P}\left[Y\ge k-1\right]}{\mathbb{E}\left[Y| Y\ge k\right] } . \]
In particular, we have 
\[ \frac{\mathbb{P}\left[Y\ge k\right]}{\mathbb{P}\left[Y\ge k-1\right]} \le \frac{\mu}{k} . \]
\end{lemma}
\begin{proof}
Notice that for all $i\ge 1$, we have 
\[ i\cdot\mathbb{P}\left[Y=i\right] =\mu\cdot \mathbb{P}\left[Y=i-1\right] .\]
This implies that
 \[ \mathbb{E}\left[Y| Y\ge k\right] = \sum_{i=k}^{\infty} i\cdot \frac{\mathbb{P}\left[Y= i\right]}{\mathbb{P}\left[Y\ge k\right]} = \mu\cdot\sum_{i=k}^{\infty} \frac{\mathbb{P}\left[Y= i-1\right]}{\mathbb{P}\left[Y\ge k\right]}  = 
  \mu\cdot \frac{\mathbb{P}\left[Y\ge k-1\right]}{\mathbb{P}\left[Y\ge k\right] } \]
and the first statement follows. The second statement follows from the last equation, upon observing that  $\mathbb{E}\left[Y| Y\ge k\right]\ge k$. 
\end{proof}

We are now ready to prove the first statement of Theorem \ref{poibound}. 

\begin{proof}[Proof of Theorem \ref{poibound} (First statement)]
Notice that 
\begin{eqnarray*}
\mathbb{E}\left[Y| Y\ge k\right] &=& \sum_{i=k}^{\infty} i\cdot \frac{\mathbb{P}\left[Y= i\right]}{\mathbb{P}\left[Y\ge k\right]} \\
&=& k + \sum_{j=1}^{\infty} j \cdot \frac{\mathbb{P}\left[Y=k+j\right]}{\mathbb{P}\left[Y\ge k\right]} \\
&=& k + \sum_{j=1}^{\infty} \frac{\mathbb{P}\left[Y\ge k+j\right]}{\mathbb{P}\left[Y\ge k\right]} .
\end{eqnarray*}
Now notice that the second statement of Lemma \ref{poilem:1} implies 
\[ \frac{\mathbb{P}\left[Y\ge k+j\right]}{\mathbb{P}\left[Y\ge k\right]}  = \prod_{\ell = 0}^{j-1} \frac{\mathbb{P}\left[Y\ge k+\ell+1\right]}{\mathbb{P}\left[Y\ge k+ \ell\right]} \le  \frac{\mu^j}{\prod_{\ell =1}^{j}(k+\ell)} .\]
Since $\prod_{\ell =1}^{j}(k+\ell) \ge (k+1)^j$ we conclude
\[  \sum_{j=1}^{\infty} \frac{\mathbb{P}\left[Y\ge k+j\right]}{\mathbb{P}\left[Y\ge k\right]} \le \sum_{j=1}^{\infty} \left(\frac{\mu}{k+1}\right)^j = \frac{1}{1-\frac{\mu}{k+1}}-1 =\frac{\mu}{k+1-\mu} .  \]
and the result follows. 
\end{proof}

Recall the following, well-known, result regarding the median of a Poisson random variable (see \cite{Choi}). \\

\begin{lemma}\label{medpoi}
Let $Y\sim\text{Poi}(\mu)$. Then $\mathbb{P}\left[Y \ge \mu -\ln 2\right] \ge 1/2$. 
\end{lemma}

We now proceed with the proof of the second statement of Theorem \ref{poibound}. 

\begin{proof}[Proof of Theorem \ref{poibound} (Second statement)]
Let $Y\sim \text{Poi}(\mu)$ and fix a positive integer $k>\mu$. Let $\ell := \lfloor k-\mu \rfloor$ and notice that $\ell$ is the maximum positive integer $j$ such that $k-j\ge \mu$. 
If we iterate Lemma \ref{poilem:1} $\ell +1$ times, we obtain 
\[ \mathbb{P}\left[Y\ge k\right] = \mathbb{P}\left[Y\ge k-(\ell +1)\right]\cdot \prod_{j=0}^{\ell} \frac{\mu}{\mathbb{E}\left[Y|Y\ge k-j\right]} . \]
Since $\ell$ is a positive integer such that $k-(\ell +1) < \mu$, Lemma \ref{medpoi} implies that 
\[ \mathbb{P}\left[Y\ge k-(\ell +1)\right] \ge \frac{1}{2} . \]
The first statement of Theorem \ref{poibound} now yields  
\[ \prod_{j=0}^{\ell} \frac{\mu}{\mathbb{E}\left[Y|Y\ge k-j\right]} \ge  \prod_{j=0}^{\ell} \frac{\mu}{k-j+\frac{\mu}{k-j+1-\mu}} = \prod_{j=0}^{\ell} \frac{\mu (k-j+1-\mu)}{(k-j)(k-j+1-\mu)+\mu}.  \]
Since $k-j\ge \mu$, for $j=0,1,\ldots,\ell$, we conclude    
\[  \frac{\mu (k-j+1-\mu)}{(k-j)(k-j+1-\mu) + \mu} \ge \frac{\mu}{(k-j)+ \mu} .\]
Putting all the above together, we conclude 
\[ \mathbb{P}\left[Y\ge k\right] \ge \frac{1}{2}\cdot \prod_{j=0}^{\ell}  \frac{\mu}{(k-j)+ \mu}  \]
and the result follows. 
\end{proof}

\section{Comparisons}\label{compare}

In this section we perform pictorial comparisons between the bound given by Theorem \ref{mainthm} 
and the bound given by (\ref{Ashbound}). In Figure \ref{fig:compare} we fix the values of 
$n$ and $k$ and plot the function $f(p):= \frac{p^{2\ell}}{2}\binom{n}{\ell +1}/\binom{k}{\ell+1} -  
\frac{1}{\sqrt{8k(1-k/n)}}\cdot \text{exp}\left(-n D(k/n||p)\right)$, 
for $p\in (0,k/n)$, where $\ell$ is as in Theorem \ref{mainthm}. For moderate values of $p$ our 
bound is almost equal to the bound given by (\ref{Ashbound}) while for large values of $k$ and $p$ 
our bound is sharper. In most cases, the two bounds are rather close to each other.

\begin{figure}[htb!]
\subfloat[][]{\includegraphics[scale=0.22]{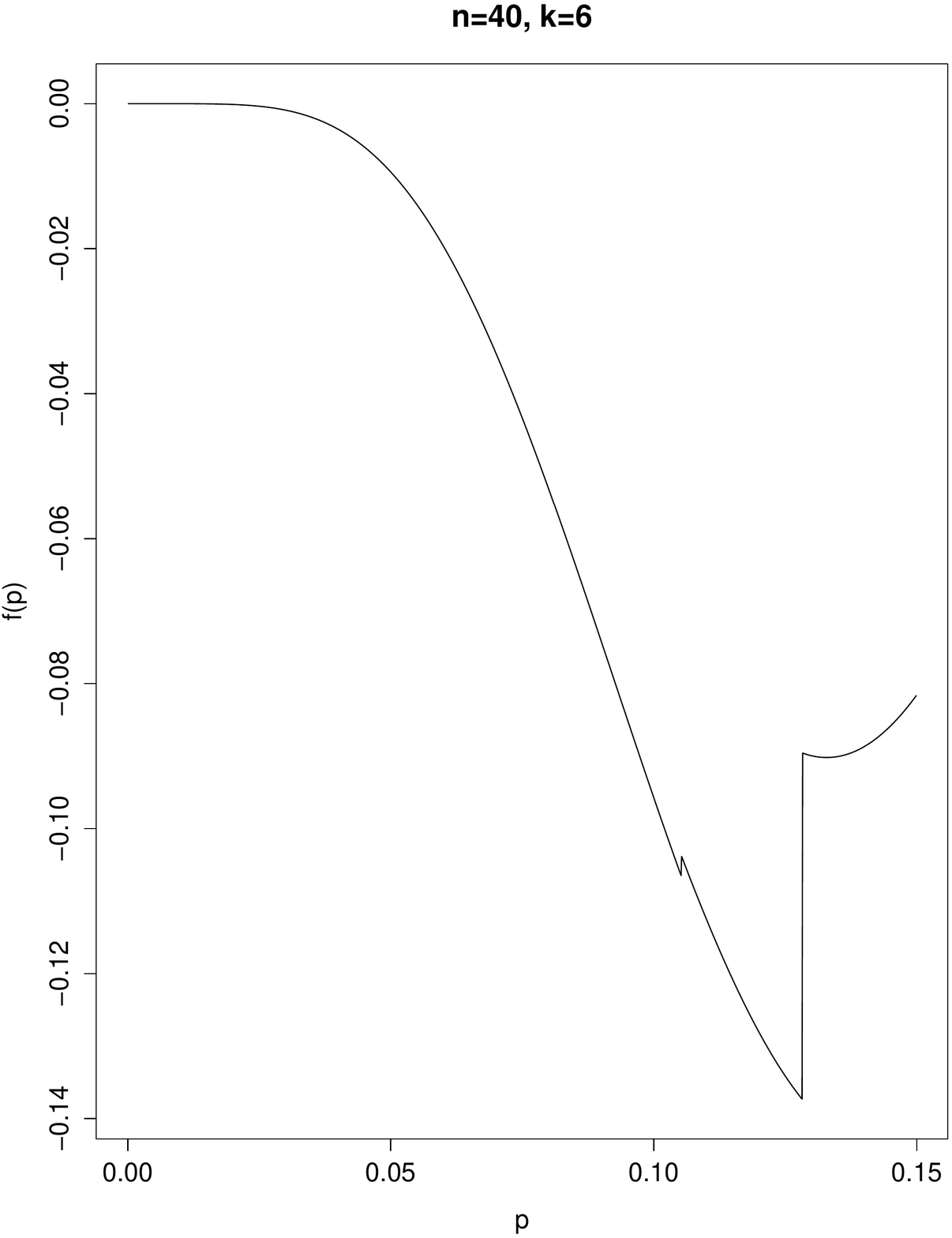}}
  \subfloat[][]{\includegraphics[scale=0.22]{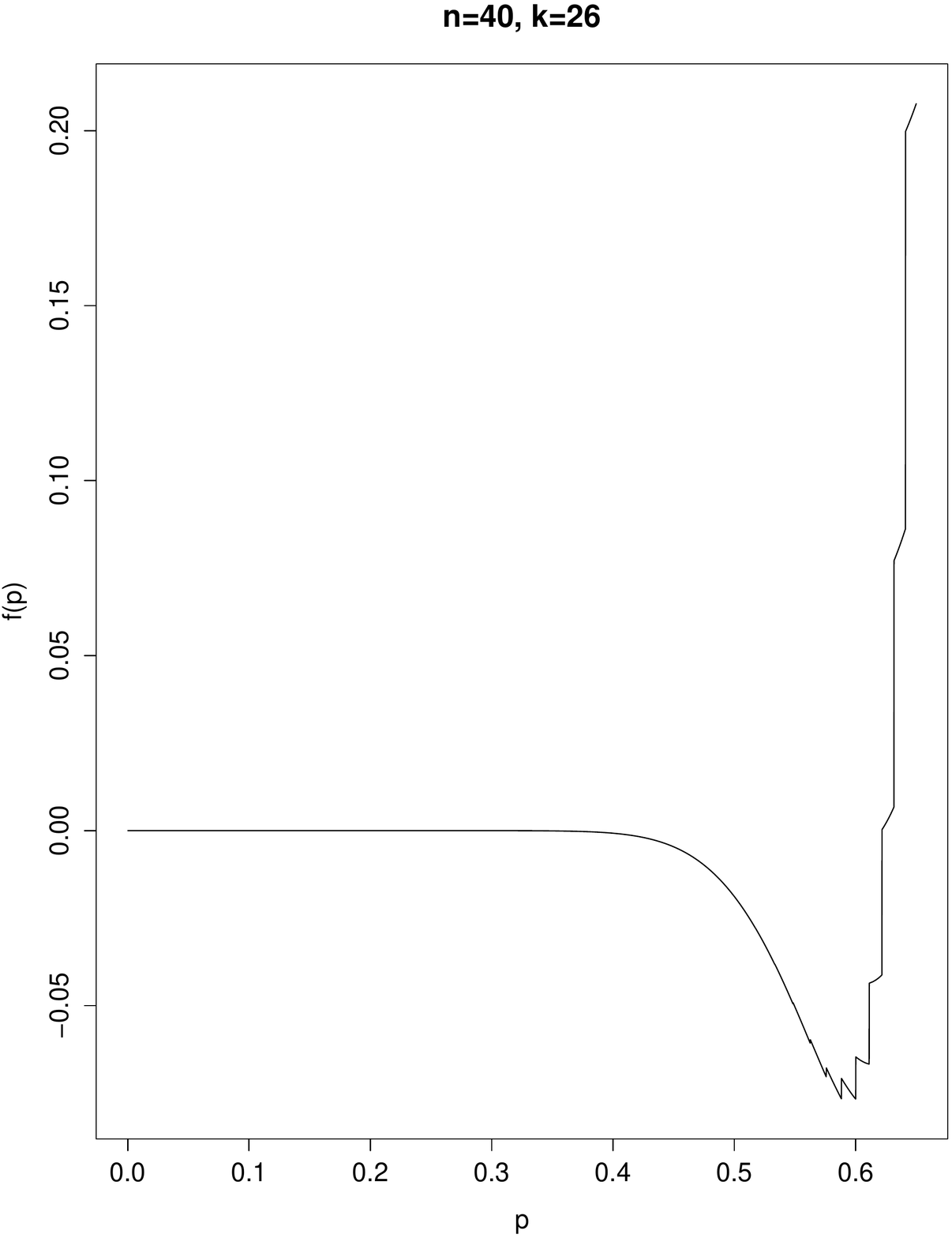}}
   \subfloat[][]{\includegraphics[scale=0.22]{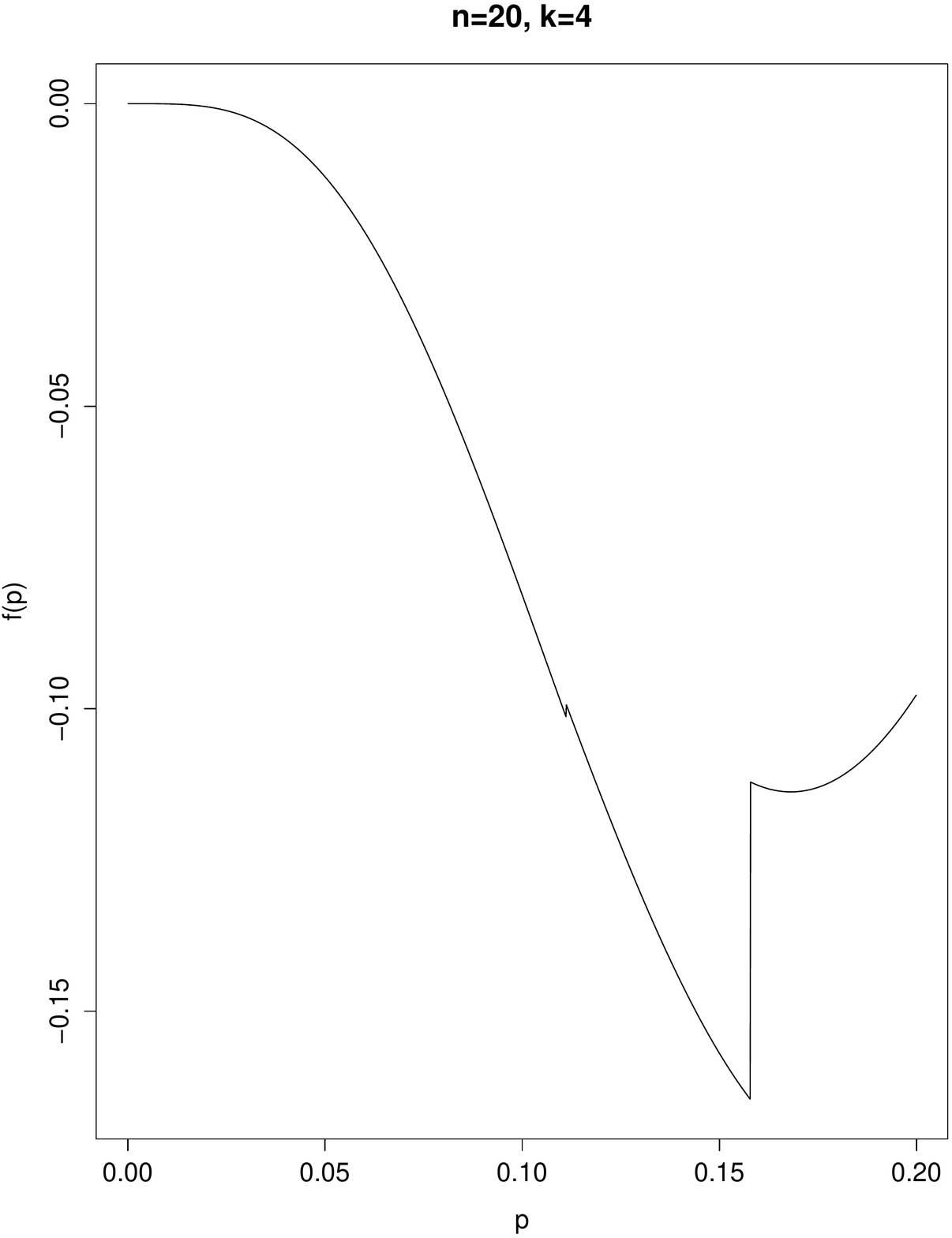}}\\
   \subfloat[][]{\includegraphics[scale=0.22]{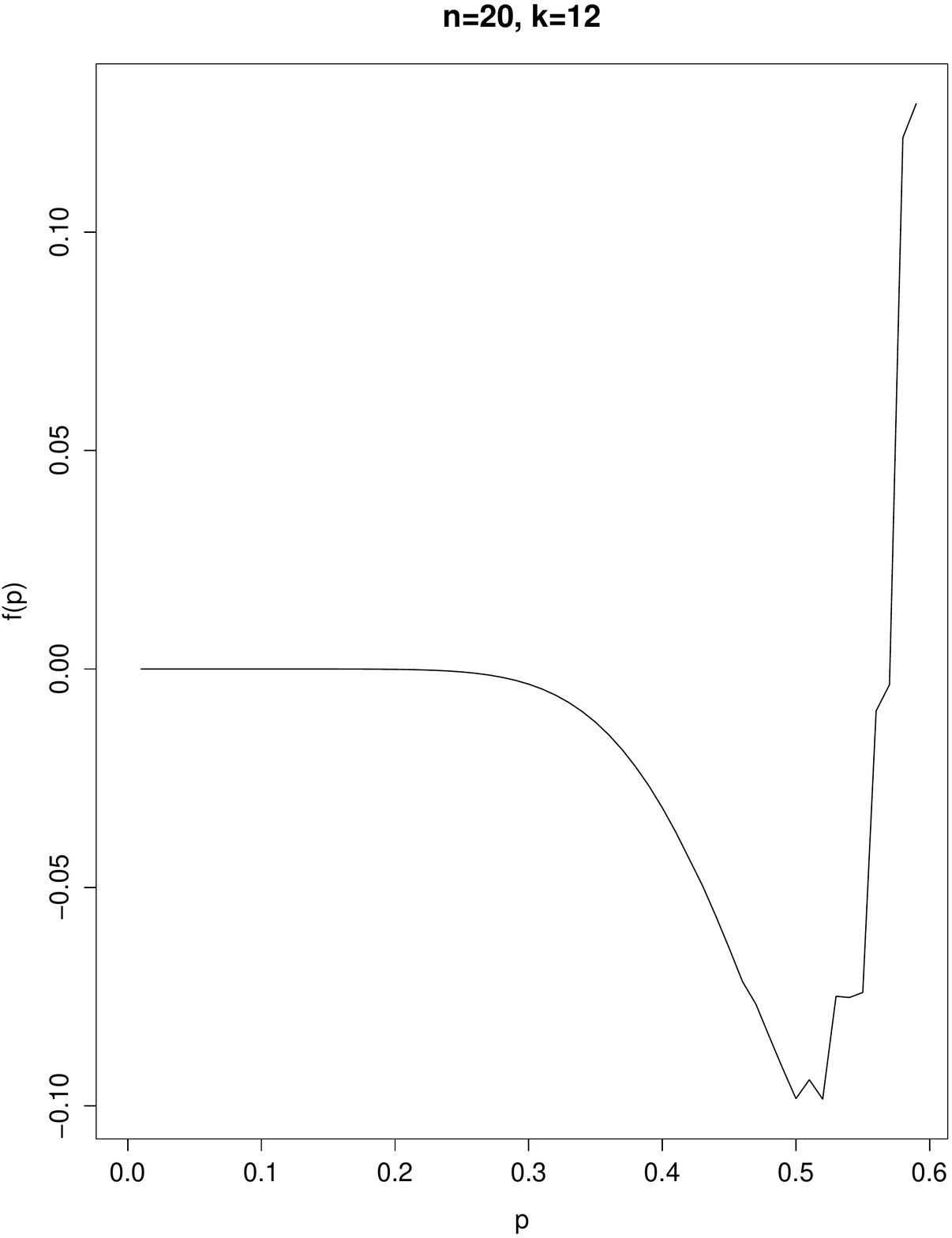}}
   \subfloat[][]{\includegraphics[scale=0.22]{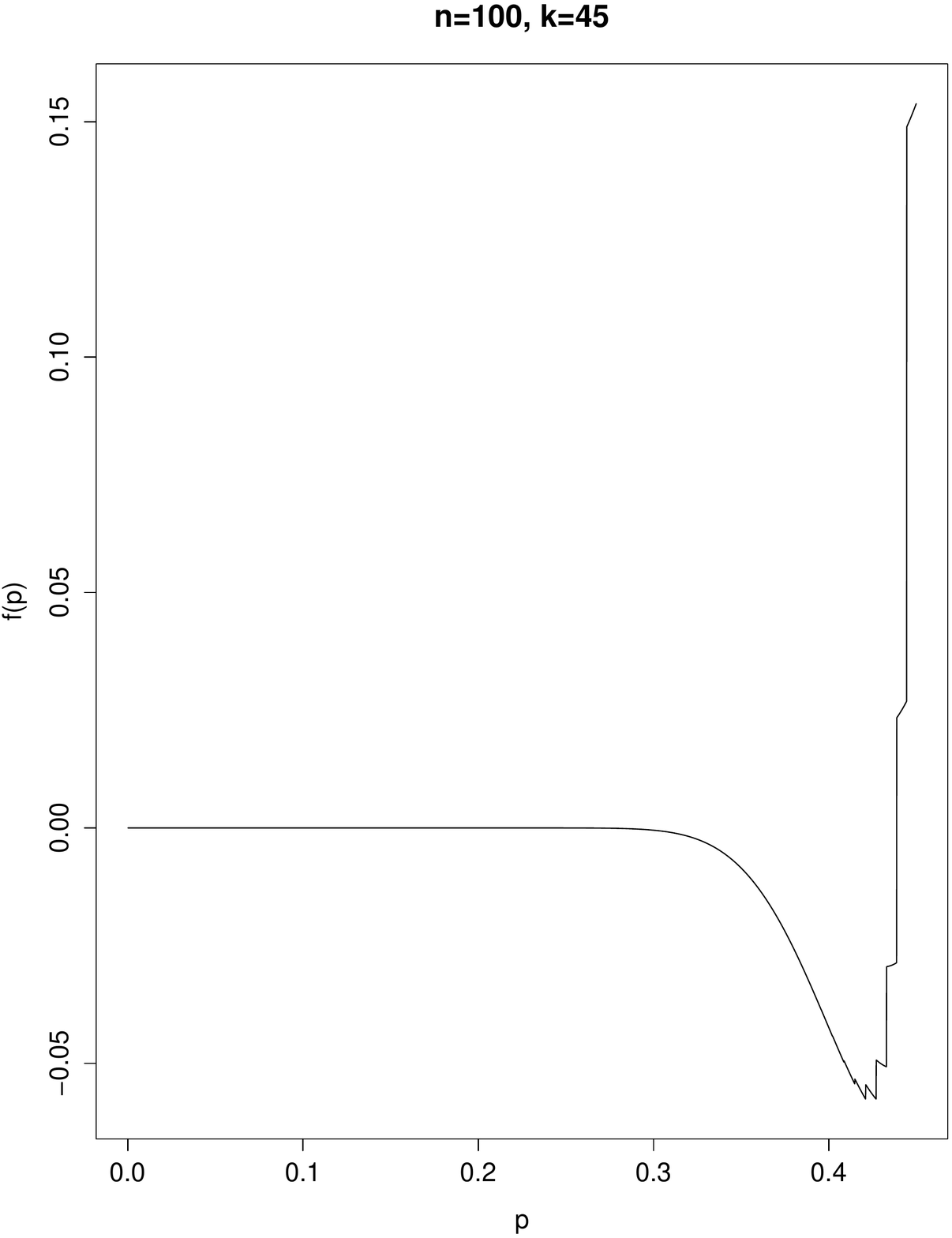}}
   \subfloat[][]{\includegraphics[scale=0.22]{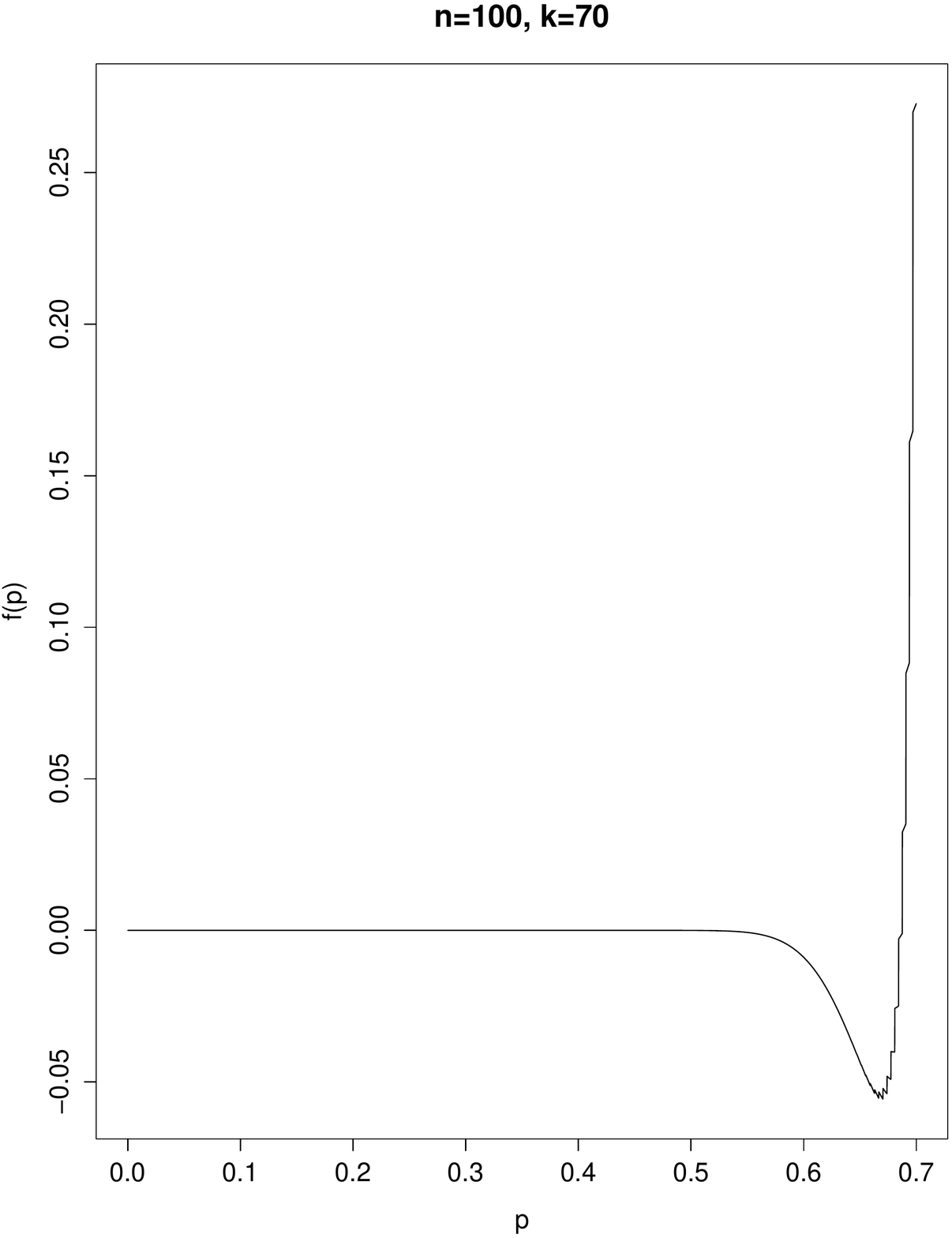}}
   \caption{Comparison between the bound of Theorem \ref{mainthm} and the bound given by 
   Eq.(\ref{Ashbound})}\label{fig:compare}
\end{figure}

\textbf{Acknowledgements}.
I am grateful to  Henk Don for fruitful discussions and valuable comments.

\end{document}